\documentclass[psfig,amssymb,amsfonts]{amsart}
\usepackage{amssymb,amsfonts}
\newtheorem{proposition}{Proposition}
\newtheorem{theorem}{Theorem}
\newtheorem{lemma}{Lemma}

\theoremstyle{remark}
\newtheorem{remark}{Remark}

\numberwithin{equation}{section}
\begin{document}

\title{Mass Concentration for the Davey-Stewartson System}
\author{Geordie Richards}
\address{Department of Mathematics, University of Toronto, Toronto, Ontario, Canada}
\email{grichard@math.toronto.edu}
\date{September 2, 2009.}
\thanks{The author thanks J. Colliander and C. Sulem for suggestions related to this work.}
\maketitle
\vspace{-0.3in}
\begin{abstract}
This paper is concerned with the analysis of blow-up solutions to the elliptic-elliptic Davey-Stewartson system, which appears in the description of the evolution of surface water waves.  We prove a mass concentration property for $H^{1}$-solutions, analogous to the one known for the $L^{2}$-critical nonlinear Schr\"{o}dinger equation.  We also prove a mass concentration result for $L^{2}$-solutions.
\end{abstract}
\section{Introduction}
We study the elliptic-elliptic Davey-Stewartson system in
2 space dimensions with cubic nonlinearity
\begin{align}
\left\{
\begin{array}{ll}iu_{t}+\Delta u + \mathcal{L}(|u|^{2})u =0, \ \
t\geq 0, x\in \mathbb{R}^{2}
\\
u(0,x) = u_{0}(x)\in H^{1}(\mathbb{R}^{2}).
\end{array} \right.
\end{align}
Here
$\mathcal{L}= \nu I + \gamma\mathcal{B}$, $\nu=\pm1,\gamma>0,$ and $\mathcal{B}$ is the pseudo-differential operator with symbol
$\widehat{\mathcal{B}(f)}(\xi)=\frac{\xi_{1}^{2}}{\xi_{1}^{2}+\xi_{2}^{2}}\widehat{f}(\xi)$.

The more general Davey-Stewartson system is given by
\begin{align}
\left\{
\begin{array}{ll}iu_{t}+ \sigma u_{xx}+ u_{yy} + \nu|u|^{2}u - \phi_{x}u =0,
\\
\alpha \phi_{xx} + \phi_{yy} + \gamma (|u|^{2})_{x} = 0 ,
\end{array} \right.
\end{align} where $\sigma, \nu = \pm1$, $\gamma>0$, and $\alpha$ can be positive or negative.
These equations describe, to leading order, the (complex) amplitude $u$ and (real) velocity potential $\phi$ of a weakly nonlinear 3-dimensional water wave traveling predominantly in the $x$-direction (\cite{DS},\cite{PSSW}).  The constants $\sigma,\alpha,\nu$ and $\gamma$ depend on physical variables, such as the strength of gravity and surface tension, the depth of the fluid domain and the wave number of the wave packet.  Depending on the signs of $(\sigma,\alpha)$, the cases $(-,-),(-,+)(+,-)$ and $(+,+)$ are classified as hyperbolic-hyperbolic, hyperbolic-elliptic, elliptic-hyperbolic and elliptic-elliptic respectively.  Following the picture from \cite{DR} (see also \cite{AS}) the elliptic-elliptic equation describes a scenario with relatively large surface tension $T$, and a sufficiently large depth $h$ to achieve a ``sub-sonic'' flow: that is the group velocity of the wave packet does not exceed the velocity $\sqrt{gh}$ of long gravity waves.

We consider the elliptic-elliptic case of (1.2).  By rescaling $x$ and $\gamma$ we can take $\alpha=1$.  Applying the Fourier transform to the second equation in (1.2), we find
\begin{align*}
(\xi_{1}^{2}+\xi_{2}^{2})\hat{\phi} = -i\gamma\xi_{1}\widehat{|u|^{2}},
\end{align*}
and, by differentiating in $x$,
\begin{align*}
(\xi_{1}^{2}+\xi_{2}^{2})\widehat{\phi_{x}}=-i\xi_{1}(\xi_{1}^{2}+\xi_{2}^{2})\hat{\phi} = -\gamma\xi_{1}^{2}\widehat{|u|^{2}}.
\end{align*}
Thus, we can solve for $\widehat{\phi_{x}}$ in terms of $u$,
\begin{align*}
\widehat{\phi_{x}}= -\gamma\frac{\xi_{1}^{2}}{\xi_{1}^{2}+\xi_{2}^{2}}(\widehat{|u|^{2}}) = -\gamma\widehat{B(|u|^{2})}.
\end{align*}
By the Fourier inversion theorem $\phi_{x}=-\gamma B(|u|^{2})$, and the system (1.2) reduces to (1.1).

We will also consider the relaxation of (1.1) to the case $u_{0}\in L^{2}(\mathbb{R}^{2})$,
\begin{align}
\left\{
\begin{array}{ll}iu_{t}+\Delta u + \mathcal{L}(|u|^{2})u =0, \ \
t\geq 0, x\in \mathbb{R}^{2}
\\
u(0,x) = u_{0}(x)\in L^{2}(\mathbb{R}^{2}).
\end{array} \right.
\end{align}

Recall the existence theory for (1.1) and (1.3).

\begin{theorem}[Ghidaglia-Saut\,\cite{GS}]
For any $u_{0}\in L^{2}(\mathbb{R}^{2})$, $\exists\  T^{*}>0$ and a unique solution $u$ to $(1.3)$ such that $u\in C([0,T^{*});L^{2}(\mathbb{R}^{2}))\cap L^{4}((0,t)\times\mathbb{R}^{2})$ for all $t\in (0,T^{*})$.  Furthermore
\begin{enumerate}
\renewcommand{\labelenumi}{(\alph{enumi})}
\renewcommand{\labelenumii}{(\alph{enumii})}

\item  $T^{*}$ is maximal in the sense that if $T^{*}< \infty$, then $\|u\|_{L^{4}_{[0,T^{*}]\times \mathbb{R}^{2}}}= \infty$.
\item  The mass of the solution is conserved, that is
\begin{align*}
\|u(t)\|_{2}=\|u_{0}\|_{2}
\end{align*}
for all $t\in (0,T^{*})$.
\end{enumerate}
\end{theorem}

\begin{theorem}[Ghidaglia-Saut\,\cite{GS}]
If $u_{0}\in H^{1}(\mathbb{R}^{2})$, the corresponding solution $u$ to $(1.1)$ satisfies $u\in C([0,T^{*});H^{1}(\mathbb{R}^{2}))\cap C^{1}([0,T^{*});H^{-1}(\mathbb{R}^{2}))$.  Furthermore
\begin{enumerate}
\renewcommand{\labelenumi}{(\alph{enumi})}
\renewcommand{\labelenumii}{(\alph{enumii})}
\item  $T^{*}$ is maximal in the sense that if $T^{*}< \infty$, then $\lim_{t\uparrow T^{*}}\|u(t)\|_{H^{1}(\mathbb{R}^{2})}=\infty$.
\item  The mass and energy of the solution are conserved, that is
\begin{align*}
\|u(t)\|_{2}=\|u_{0}\|_{2} \\
E(u(t))=E(u_{0})
\end{align*}
for all $t\in (0,T^{*})$, where
\begin{align*}
E(u)=\frac{1}{2}\int|\nabla u|^{2}dx - \frac{1}{4}
\int \mathcal{L}(|u|^{2})|u|^{2}dx.
\end{align*}
\end{enumerate}
\end{theorem}

We review what is known (see also \cite{Ci},\cite{PSSW}) about standing wave and blow-up solutions to (1.1).
There are standing wave solutions to (1.1) of the form
$u(t,x)=v(x)e^{it}$.

\begin{theorem}[Papanicolaou et al\,\cite{PSSW}]  Taking $\nu=1$ \emph{(}focusing\emph{)}, we have the optimal estimate
\begin{align}
\int \mathcal{L}(|u|^{2})|u|^{2}dx \leq C_{opt}\|\nabla
u\|_{2}^{2}\|u\|_{2}^{2}.
\end{align}
 Furthermore $C_{opt}=\frac{2}{\|R\|_{2}^{2}}$ for some $R\in H^{1}(\mathbb{R}^{2})$ such that $R(x)>0$ $\forall \,x\in\mathbb{R}^{2}$, and
$u(t,x)=R(x)e^{it}$ solves $(1.1)$.
\end{theorem}

\begin{remark}
The uniqueness of ground state standing wave solutions to (1.1) (ie. solutions $u(t,x)=R(x)e^{it}$ with $R(x)>0$ $\forall \,x\in\mathbb{R}^{2}$) is an open problem (see \cite{Oh} for discussion).  Thus the $L^{2}$ norm ($\|R\|_{2}$) for such a solution is not apriori well-defined.
Rather, the inequality (1.4) is sharp \cite{PSSW}, and the optimal constant $C_{opt}>0$ is therefore unique.
\end{remark}

All sufficiently well-localized negative energy initial data blow-up in finite time:

\begin{theorem}[Ghidaglia-Saut\,\cite{GS}]
Let $\Sigma := \{v\in H^{1}(\mathbb{R}^{2}):|x|v\in L^{2}(\mathbb{R}^{2})\}$.  The following holds true
\begin{enumerate}
\renewcommand{\labelenumi}{(\alph{enumi})}
\renewcommand{\labelenumii}{(\alph{enumii})}
\renewcommand{\labelenumiii}{(\alph{enumiii})}
\item  If $-\nu\geq \gamma$, all solutions of $(1.1)$ are global in time.
\item  An element $v\in \Sigma$ satisfying $E(v)<0$ exists if and only if $-\nu< \gamma$.
\item  If $u_{0}\in \Sigma$ satisfies $E(u_{0})<0$, the corresponding solution $u$ to $(1.1)$ blows up in finite time.
\end{enumerate}
\end{theorem}

\begin{remark}
Part (c) of Theorem 4 follows from the following virial identity, first obtained in \cite{AS},
\begin{align}
\int|x|^{2}|u(t)|^{2}dx = 4E(u_{0})t^{2} + ct + \int|x|^{2}|u_{0}|^{2}dx.
\end{align}
For a non-zero solution $u$ to (1.1), the left-hand side of (1.5) is manifestly positive.  But, if $E(u_{0})<0$, the right-hand side of (1.5) becomes negative in finite time.  The corresponding solution $u$ cannot, therefore, be global in time.  By part (a) of Theorem 2, $u$ is a blow-up solution.
\end{remark}

\begin{remark}
A scaling argument yields a lower bound on the blow-up rate for solutions to (1.1) with $T^{*}<\infty$:
\begin{align}
\|\nabla u(t)\|_{2} \geq  \frac{C}{\sqrt{T^{*}-t}}.
\end{align}
If $u(t,x)$ solves (1.1) on $[0,T^{*})$, then $\frac{1}{\lambda}u(\frac{t}{\lambda^{2}},\frac{x}{\lambda})$, for $\lambda>0$, solves (1.1) on $[0,\lambda^{2}T^{*})$.  To prove (1.6),
assume that $u$ is a finite time blow-up solution to (1.1).  For fixed $t\in[0,T^{*})$, consider the solution to (1.1) formed by the rescaling
\begin{align*}
v^{t}(\tau,x)=\|\nabla u(t)\|_{2}^{-1}u(t+\|\nabla u(t)\|_{2}^{-2}\tau,\|\nabla u(t)\|_{2}^{-1}x).
\end{align*}
We compute $\|v^{t}(0)\|_{H^{1}(\mathbb{R}^{2})}=C$ independent of $t$.  From the fixed point argument used to prove Theorem 2 (see \cite[p. 487]{GS}), $\exists \, \tau_{0}=\tau_{0}(\|v^{t}(0)\|_{H^{1}(\mathbb{R}^{2})})$ independent of $t$ such that $v^{t}$ is defined on $[0,\tau_{0}]$.  This implies that $t+ \|\nabla u(t)\|_{2}^{-2}\tau_{0} \leq T^{*}$, and (1.6) is proven.
\end{remark}

Global existence in $H^{1}(\mathbb{R}^{2})$ for solutions to (1.1) with $\|u_{0}\|_{2}<\sqrt{\frac{2}{C_{opt}}}$ is a corollary of Theorem 3 \cite{PSSW}.  This result is optimal (with respect to mass) due to an explicit blow-up solution induced by a pseudo-conformal invariance:  If $u$ solves (1.1) for $t\in[1,\infty)$, then
\begin{align}
 pc[u](t,x):= e^{-i\frac{|x|^{2}}{4t}}\frac{1}{|t|}\overline{u}(-\frac{1}{t},\frac{x}{t})
\end{align}
solves (1.1) for $t\in[-1,0)$.  We can apply this symmetry to the ground state standing wave solution $u_{R}(t,x):=R(x)e^{it}$, with $\|R\|_{2}=\sqrt{\frac{2}{C_{opt}}}$,  to find a blow-up solution given by
\begin{align}pc[u_{R}](t,x) = e^{-i\frac{|x|^{2}}{4t}+\frac{i}{t}}\frac{1}{|t|}R(\frac{x}{t}).\end{align}  This solution to (1.1) satisfies $ \|pc[u_{R}]\|_{2} = \|R\|_{2}=\sqrt{\frac{2}{C_{opt}}}$ and $\|\nabla pc[u_{R}](t)\|_{2} \sim \frac{1}{|t|}$ as $t\uparrow 0$; $pc[u_{R}]$ is a blow-up solution with the minimal mass $\sqrt{\frac{2}{C_{opt}}}$.

Taking $\gamma=0$, $\nu =1$, (1.1) reduces to the focusing $L^{2}$-critical (cubic in $\mathbb{R}^{2}$) nonlinear Schr\"{o}dinger equation (NLS),
\begin{align}
\left\{
\begin{array}{ll}iu_{t}+\Delta u  =-|u|^{2}u, \ \
t\geq 0, x\in \mathbb{R}^{2}
\\
u(0,x) = u_{0}(x)\in H^{1}(\mathbb{R}^{2}).
\end{array} \right.
\end{align}
Standing wave and blow-up solutions to (1.9) have been exposed in a series of works.
There is a \textit{unique} ground state standing wave solution to (1.9), given by $u_{Q}(t,x):=Q(x)e^{it}$.
When $\gamma=0$, Theorem 3 is a well-known sharp inequality \cite{We}.  This inequality is optimized by the
ground state profile $Q(x)$, and consequently, for initial data satisfying $\|u_{0}\|_{2}<\|Q\|_{2}$, the corresponding solution to (1.9) is global-in-time.  Applying the pseudo-conformal symmetry (1.7) to the ground state standing wave solution $u_{Q}(t,x)=Q(x)e^{it}$, we obtain a blow-up solution $pc[u_{Q}](t,x)$ with the minimal mass $\|u_{0}\|_{2}=\|Q\|_{2}$.  Up to symmetries of (1.9), $pc[u_{Q}](t,x)$ is the \textit{only} minimal mass blow-up solution \cite{M1}.  All blow-up solutions to (1.9) concentrate mass.  More precisely, they concentrate at least the mass $\|Q\|_{2}$ into a parabolically shrinking window as $t$ approaches the blow-up time $T^{*}$ \cite{BV},\cite{HK},\cite{MT}.  For solutions with mass $\|u_{0}\|_{2}$ slightly larger than $\|Q\|_{2}$, the existence of two distinct blow-up regimes, precise norm explosion rates and asymptotic profile properties have been established \cite{MR}.

This paper is motivated by asymptotic and numerical analyses \cite{PSSW} which reveal a similar blow-up phenomenon for (1.1) and (1.9).  We prove a mass concentration property for blow-up solutions to (1.1), analogous to the one known for (1.9), by adapting methods from \cite{HK}.  We use a profile decomposition and Theorem 3 to prove a compactness property for bounded sequences in $H^{1}(\mathbb{R}^{2})$.  This property is then applied to spatially rescaled snapshots (in time) of a blow-up solution $u$ to prove mass concentration.  Here are the results we obtain:

\begin{theorem}
Let $\{v_{n}\}_{n=1}^{\infty}$ be a bounded sequence in
$H^{1}(\mathbb{R}^{2})$ such that
\begin{align}
\limsup_{n\rightarrow \infty}\|\nabla v_{n}\|^{2}_{2} \leq M^{2}, \ \
\limsup_{n\rightarrow \infty}
\int\mathcal{L}(|v_{n}|^{2})|v_{n}|^{2}dx \geq m^{4},
\end{align}
for some $0<m,M<\infty$.  Then there exists
$\{x_{n}\}_{n=1}^{\infty} \subset \mathbb{R}^{2}$ such that up to a
subsequence
\begin{align*}
v_{n}(\cdot + x_{n})\rightharpoonup V \text{ weakly in $H^{1}(\mathbb{R}^{2})$ with } \|V\|_{2}
\geq \frac{m^{2}}{M\sqrt{C_{opt}}}.
\end{align*}
\end{theorem}

\begin{theorem}
Let $u$ be a solution of $(1.1)$ which blows up in finite time $T^{*}>0$,
and $\lambda(t)>0$ any function such that $\lambda (t)\|\nabla
u(t)\|_{2} \rightarrow +\infty$ as $t\uparrow T^{*}$.  Then, $\exists
\ y(t)\in \mathbb{R}^{2}$ such that
\begin{align*}
\liminf_{t \uparrow T^{*}}\int_{|x-y(t)|\leq
\lambda(t)}|u(t,x)|^{2}dx \geq \frac{2}{C_{opt}}.
\end{align*}
\end{theorem}

\begin{remark}
Recall from Theorem 3 that $\frac{2}{C_{opt}}=\|R\|_{2}^{2}$ for a function $R\in H^{1}(\mathbb{R}^{2})$ such that $u(t,x)=R(x)e^{it}$ solves (1.1) and $R(x)>0 \ \forall \,x\in \mathbb{R}^{2}$.
\end{remark}

\begin{remark}

If $\lambda(t)\gtrsim (T^{*}-t)^{1/2-\epsilon}$ for a small $\epsilon>0$, then by (1.6), $\lambda(t)\|\nabla u(t)\|_{2} \geq C(T^{*}-t)^{-\epsilon} \rightarrow \infty$ as $t\uparrow T^{*}$.  Thus Theorem 6 holds for such a function $\lambda(t)$, and implies a (nearly) parabolic mass concentration effect for blow-up solutions to (1.1).  Observe that the blow-up solution (1.8) concentrates mass within a smaller conic window at time $T^{*}=0$: $\lambda(t)\gtrsim|t|^{1-}$.

\end{remark}

Bourgain has proven \cite{B1} a mass concentration property for solutions to cubic NLS posed in $L^{2}(\mathbb{R}^{2})$ with a finite lifespan ($T^{*}<\infty$).
We adapt Bourgain's proof to (1.3), and obtain the following result:

\begin{theorem}
Let $u$ be a solution of $(1.3)$ which satisfies $T^{*}<\infty$.  Take $M:=\|u_{0}\|_{2}$.  Then there exist constants $\eta(M),C>0$  such that
\begin{align}
\limsup_{t\uparrow T^{*}} \sup_{\substack{\text{squares}\  Q\subset \mathbb{R}^{2} \\
\ \text{of sidelength} \ \ell(Q)<C(T^{*}-t)^{1/2}}} \Big( \int_{Q}|u(t,x)|^{2}dx\Big)^{1/2} > \eta(M)>0.
\end{align}
\end{theorem}

\begin{remark}
Refinements of Theorem 7 relating the window size of mass concentration and rate of explosion of the $L^{4}_{[0,t]\times \mathbb{R}^{2}}$ norm also hold \cite{CR}.  That is, the window of mass concentration will have sidelength $(T^{*}-t)^{\frac{1+\beta}{2}}$ if and only if the $L^{4}_{[0,t]\times \mathbb{R}^{2}}$ norm grows in time no slower than $(T^{*}-t)^{-\beta}$.
\end{remark}
In Section 2 we use Theorem 5 to prove the mass concentration result for $H^{1}(\mathbb{R}^{2})$ solutions (Theorem 6). The proof of Theorem 5 is presented in Section 3.  Finally, the mass concentration result for $L^{2}(\mathbb{R}^{2})$ solutions (Theorem 7) is proven in Section 4.

\section{Mass concentration for $H^{1}$-solutions}
\noindent $\textit{Proof of Theorem 6.}$
Assume that $H^{1}(\mathbb{R}^{2})\ni u_{0}\mapsto u(t)$ is any finite time blow-up solution to (1.1).  We introduce a shrinking parameter which encodes the core size of the blow-up region:
\begin{align*} \rho (t) = \frac{1}{\|\nabla
u(t)\|_{2}},\end{align*}
and rescale
\begin{align}  v(t,x)=\rho u(t,\rho x).\end{align}
Let
$\{t_{n}\}_{n=1}^{\infty}$ be an arbitrary sequence such that
$t_{n}\uparrow T^{*}, \rho_{n}=\rho(t_{n})$ and $v_{n}=v(t_{n},\cdot)
$.  The mass of $u$ is invariant under both the rescaling (2.1), and the flow of (1.1), and therefore
\begin{align*}
\|v_{n}\|_{2}=\|u(t_{n})\|_{2}=\|u_{0}\|_{2}.
\end{align*}
We then compute, by choice of $\rho$, that
\begin{align*}
\|\nabla v_{n}\|_{2}=1.
\end{align*}
From linearity of $\mathcal{L}$, conservation of
energy and blow-up, we have
\begin{align*}
E(v_{n}) &= \frac{1}{2}-\frac{1}{4}\int \mathcal{L}(|v_{n}|^{2})|v_{n}|^{2}dx    \\
&= \frac{1}{2}\int |\nabla v_{n}|^{2}dx-\frac{1}{4}\int \mathcal{L}(|v_{n}|^{2})|v_{n}|^{2}dx    \\
&= \frac{1}{2}\int |\nabla (\rho_{n}u(t_{n},\rho_{n}x))|^{2}dx-\frac{1}{4}\int \mathcal{L}(|\rho_{n}u(t_{n},\rho_{n}x)|^{2})|\rho_{n}u(t_{n},\rho_{n}x)|^{2}dx    \\
&= \rho_{n}^{4}\Big(\frac{1}{2}\int|(\nabla u)(t_{n},\rho_{n}x)|^{2}dx-\frac{1}{4}\int \mathcal{L}(|u(t_{n})|^{2})(\rho_{n}x)|u(t_{n},\rho_{n}x)|^{2}dx \Big)   \\
&= \rho_{n}^{2}E(u(t_{n}))   \\
&= \frac{1}{\|\nabla
u(t_{n})\|^{2}_{2}}E(u_{0}) \rightarrow 0, \ \text{ as }n\rightarrow
\infty.
\end{align*}
In particular,
\begin{align*}
\lim_{n\rightarrow \infty} \int
\mathcal{L}(|v_{n}|^{2})|v_{n}|^{2}dx = 2.
\end{align*}
The sequence $\{v_{n}\}$ satisfies the hypotheses (1.10) of Theorem 5 with
\begin{align*}
m^4=2, \ \ M=1.
\end{align*}
Let us assume Theorem 5 holds true (we prove Theorem 5 in the next section).  By Theorem 5 there is a sequence $\{x_{n}\}_{n=1}^{\infty} \subset \mathbb{R}^{2}$
and a profile $V\in H^{1}(\mathbb{R}^{2})$, with $\|V\|_{2}\geq
\frac{m^{2}}{M\sqrt{C_{opt}}}=\sqrt{\frac{2}{C_{opt}}}$, such that, up to a subsequence
\begin{align*}
v_{n}=\rho_{n}u(t_{n},\rho_{n}\cdot +x_{n})\rightharpoonup V \text{
weakly in }H^{1}(\mathbb{R}^{2}).
\end{align*}
Then for any $A>0$, $\chi_{|x|\leq A}\rho_{n}u(t_{n},\rho_{n}\cdot +x_{n}) \rightharpoonup \chi_{|x|\leq A} V$ weakly in $L^{2}(\mathbb{R}^{2})$.  By lower semi-continuity of the norm in the weak limit,
\begin{align*}
\liminf_{n\rightarrow \infty}\int_{|x|\leq
A}\rho_{n}^{2}|u(t_{n},\rho_{n}x+x_{n})|^{2}dx \geq \int_{|x|\leq
A}|V|^{2}dx.
\end{align*}
With a change of variables we have
\begin{align*}
\liminf_{n\rightarrow \infty}\int_{|x-x_{n}|\leq \rho_{n}A}|u(t_{n},x)|^{2}dx  \geq \int_{|x|\leq
A}|V|^{2}dx.
\end{align*}
From the assumption of Theorem 6, that $\lambda(t)\|\nabla u(t)\|_{2}\rightarrow \infty$ as $t\uparrow T^{*}$, $$\frac{\lambda(t_{n})}{\rho_{n}} = \lambda(t_{n})\|\nabla u(t_{n})\|_{2} \rightarrow \infty.$$  Thus $\lambda(t_{n})\geq \rho_{n}A$ for $n$ sufficiently large, and
\begin{align*}
\liminf_{n\rightarrow \infty}\sup_{y\in
\mathbb{R}^{2}}\int_{|x-y|\leq \lambda(t_{n})}|u(t_{n},x)|^{2}dx
&\geq
\liminf_{n\rightarrow \infty}\sup_{y\in
\mathbb{R}^{2}}\int_{|x-y|\leq \rho_{n}A}|u(t_{n},x)|^{2}dx  \\
&\geq
\liminf_{n\rightarrow \infty}\int_{|x-x_{n}|\leq \rho_{n}A}|u(t_{n},x)|^{2}dx  \\
&\geq  \int_{|x|\leq A}|V|^{2}dx \ \ \text{  for every }A>0.
\end{align*}
Taking $A$ to infinity,
\begin{align*} \liminf_{n\rightarrow \infty}\sup_{y\in
\mathbb{R}^{2}}\int_{|x-y|\leq \lambda(t_{n})}|u(t_{n},x)|^{2}dx
\geq  \|V\|_{2}^{2} \geq \frac{2}{C_{opt}}.
\end{align*}
Since the sequence $\{t_{n}\}_{n=1}^{\infty}$ is arbitrary,
\begin{align}\liminf_{t\uparrow T^{*}}\sup_{y\in
\mathbb{R}^{2}}\int_{|x-y|\leq \lambda(t)}|u(t,x)|^{2}dx \geq
\frac{2}{C_{opt}}.
\end{align}
For every $t\in[0,T^{*}),$ the function $y\mapsto \int_{|x-y|\leq
\lambda(t)}|u(t,x)|^{2}dx $ is continuous and vanishes at infinity.
The supremum in (2.2) is therefore a maximum and there exists $y(t)\in \mathbb{R}^{2}$ such that
\begin{align}
\sup_{y\in \mathbb{R}^{2}}\int_{|x-y|\leq \lambda(t)}|u(t,x)|^{2}dx
=\int_{|x-y(t)|\leq \lambda(t)}|u(t,x)|^{2}dx.
\end{align}
Combining (2.2) and (2.3),
\begin{align*}
\liminf_{t \uparrow T^{*}}\int_{|x-y(t)|\leq
\lambda(t)}|u(t,x)|^{2}dx \geq \frac{2}{C_{opt}}.
\end{align*}
This completes the proof of Theorem 6 under the assumption that Theorem 5 holds true.
\qed
\section{A Compactness Property}
The proof of Theorem 5 relies on a profile decomposition for
bounded sequences in $H^{1}(\mathbb{R}^{2})$.
\begin{proposition}\cite{HK}
Let $\{v_{n}\}_{n=1}^{\infty}$ be a bounded sequence in
$H^{1}(\mathbb{R}^{2})$.  Then there exists a subsequence of
$\{v_{n}\}_{n=1}^{\infty} ($still denoted
$\{v_{n}\}_{n=1}^{\infty})$, a family $\{x^{j}_{n}\}_{j=1}^{\infty}$
for each $n$, and a bounded sequence $\{V^{j}\}_{j=1}^{\infty}
\subset H^{1}(\mathbb{R}^{2})$, such that
\begin{enumerate}
\renewcommand{\labelenumi}{(\alph{enumi})}
\renewcommand{\labelenumii}{(\alph{enumii})}
\item $\forall \  i \neq j$,
\begin{align}|x^{i}_{n}-x^{j}_{n}|\rightarrow \infty \ \text{as}\  n\rightarrow \infty.
\end{align}
\item $\forall \ k\geq 1$ and $\forall \,x \in \mathbb{R}^{2}$
\begin{align} v_{n}(x)= \sum_{j=1}^{k}V^{j}(x-x_{n}^{j}) + v^{k}_{n}(x)
\end{align}
with
\begin{align}\limsup_{n \rightarrow \infty}\|v_{n}^{k}\|_{p}\rightarrow 0
\ \text{as}\  k\rightarrow \infty \ \text{for every}\  p\in (2,\infty).
\end{align}
\end{enumerate}
Moreover, as $n\rightarrow \infty$,
\begin{align}
\|v_{n}\|_{2}^{2}&= \sum_{j=1}^{k}\|V^{j}\|_{2}^{2} +
\|v_{n}^{k}\|_{2}^{2} + o(1), \\
\|\nabla v_{n}\|_{2}^{2}&= \sum_{j=1}^{k}\|\nabla V^{j}\|_{2}^{2} +
\|\nabla v_{n}^{k}\|_{2}^{2} + o(1).
\end{align}
Here $o(1)$ represents terms that go to zero as $n\rightarrow \infty$.
\end{proposition}
\begin{proof}(Outline, a complete proof is given in \cite[p. 2823]{HK}, Proposition 3.1.  See also \cite{MV} for related asymptotic compactness modulo symmetries results.)
For any sequence of functions $\textbf{u}=\{u_{n}\}_{n=1}^{\infty}\subset H^{1}(\mathbb{R}^{2})$, let $\mathcal{V}(\textbf{u})$ denote the set of weak limits of subsequences of $\{u_{n}(\cdot+x_{n})\}$ with $\{x_{n}\}_{n=1}^{\infty}\subset \mathbb{R}^{2}$.
Also take $\eta(\textbf{u}):= \sup \{\|V\|_{H^{1}(\mathbb{R}^{2})}:V\in \mathcal{V}(\textbf{u})\}$.

Let $\textbf{v}^{0}=\{v_{n}\}_{n=1}^{\infty}$ be a bounded sequence in $H^{1}(\mathbb{R}^{2})$.
The sequence of functions $\{V^{j}\}_{j=1}^{\infty}$, points $\{x^{j}_{n}\}_{n=1}^{\infty}$, and remainders
$\textbf{v}^{k}:=\{v_{n}^{k}\}_{n=1}^{\infty}$ will be constructed by induction on $k$.  That is, the $k^{\text{th}}$ function $V^{k}$ will be extracted from the set $\mathcal{V}(\textbf{v}^{k-1})$ of weak limits of translates of the remainders $\textbf{v}^{k-1}$.  The translation points emerging from this selection form the required sequence $\{x^{k}_{n}\}_{n=1}^{\infty}$.

If $\eta(\textbf{v}^{0})=0$ we can take $V^{j}\equiv 0$, $\forall \,j$.  Otherwise we choose $V^{1}\in \mathcal{V}(\textbf{v}^{0})$ such that $$ \|V^{1}\|_{H^{1}(\mathbb{R}^{2})} \geq \frac{1}{2}\eta(\textbf{v}^{0})>0.$$
By definition, there exists some sequence $\{x_{n}^{1}\}_{n=1}^{\infty}\subset \mathbb{R}^{2}$, such that up to a subsequence
\begin{align*} v_{n}(\cdot+x_{n}^{1})\rightharpoonup V^{1} \ \text{weakly in} \ H^{1}(\mathbb{R}^{2}).\end{align*}
We set
\begin{align*} v_{n}^{1}=v_{n}-V^{1}(\cdot-x_{n}^{1})\end{align*} for all $n\geq1$.
Then $v_{n}^{1}(\cdot+x^{1}_{n})\rightharpoonup 0$ weakly in $H^{1}(\mathbb{R}^{2})$, and this gives
\begin{align*}
\|v_{n}\|_{2}^{2} = \|v_{n}(\cdot+x_{n})\|_{2}^{2} &= \|v_{n}^{1}(\cdot+x_{n}) + V^{1}(\cdot)\|_{2}^{2}  \\
 &=  \|v_{n}^{1}(\cdot+x_{n})\|_{2}^{2} + 2Re\langle v_{n}^{1}(\cdot+x_{n}),V^{1}\rangle + \|V^{1}\|_{2}^{2} \\
 &=  \|v_{n}^{1}\|_{2}^{2} + \|V^{1}\|_{2}^{2} + o(1).
\end{align*}
This proves (3.4) for $k=1$, and (3.5) follows with the same reasoning.

We replace $\textbf{v}^{0}$ by $\textbf{v}^{1}=\{ v_{n}^{1}\}_{n=1}^{\infty}$, and repeat this process.
If $V^{2}\neq 0$, (3.1) holds for $i=1$, $j=2$.  Otherwise, up to a subsequence $x_{n}^{j}-x_{n}^{i}\rightarrow x_{0}\in \mathbb{R}^{2}$.  Writing $$ v_{n}^{1}(\cdot+x^{2}_{n})=v_{n}^{1}(\cdot+(x^{2}_{n}-x^{1}_{n})+x^{1}_{n}),$$then as $v_{n}^{1}(\cdot+x^{1}_{n})\rightharpoonup 0$ weakly in $H^{1}(\mathbb{R}^{2})$, and $v_{n}^{1}(\cdot+x^{2}_{n})\rightharpoonup V^{2}$ weakly in $H^{1}(\mathbb{R}^{2})$, we find $V^{2}=0$, a contradiction.  Proceeding inductively, we can extract the required sequence $\{V^{j}\}_{j=1}^{\infty}$ and family of sequences $\{x^{j}_{n}\}_{n=1}^{\infty}$.  Convergence of the series $\sum_{j=1}^{\infty}\|V^{j}\|_{H^{1}(\mathbb{R}^{2})}^{2}$ (by $(3.4),(3.5)$) implies that $\|V^{j}\|_{H^{1}(\mathbb{R}^{2})} \rightarrow 0$ as $j\rightarrow \infty$.  By construction we have $\eta(\textbf{v}^{k})\leq 2\|V^{k-1}\|_{H^{1}(\mathbb{R}^{2})}$, which gives us $\eta(\textbf{v}^{k})\rightarrow 0$ as $k\rightarrow \infty$.  The proof of (3.3), which relies on $\eta(\textbf{v}^{k})\rightarrow 0$ as $k\rightarrow \infty$, is omitted.
\end{proof}
We require an additional lemma describing a property of the profiles obtained in Proposition 1.
\begin{lemma}
Given bounded sequences $\{v_{n}\}_{n=1}^{\infty}, \{V_{j}\}_{j=1}^{\infty}\subset H^{1}(\mathbb{R}^{2})$ and a family of sequences $\{x_{n}^{j}\}_{j=1}^{\infty}\subset \mathbb{R}^{2}$ such that properties $(3.1)$-$(3.5)$ are satisfied, then
\begin{align}
\limsup_{n\rightarrow \infty}\int
\mathcal{L}(|v_{n}|^{2})|v_{n}|^{2}dx \leq \sum_{j=1}^{\infty}\int
\mathcal{L}(|V^{j}|^{2})|V^{j}|^{2}dx.
\end{align}
\end{lemma}
\begin{remark}
(3.6) is a modification of an inequality from \cite[p. 2825-2826]{HK}, (3.29)-(3.31), for the NLS case (1.9).  It is designed to be used with Theorem 3 to adapt the arguments of \cite{HK} to the system (1.1).
When we prove Lemma 1, and when we apply Theorem 3, we invoke properties of (1.1) distinct from (1.9).  Otherwise, our arguments are identical to those from \cite{HK}.
\end{remark}
The proof of Lemma 1 is postponed to the end of this section.
\vspace{0.2in}

\noindent $\textit{Proof of Theorem 5.}$
Given a bounded sequence $\{v_{n}\}_{n=1}^{\infty}\subset H^{1}(\mathbb{R}^{2})$ satisfying (1.10), we apply Proposition 1 to get a sequence of functions $\{V_{j}\}_{j=1}^{\infty}\subset H^{1}(\mathbb{R}^{2})$, and a sequence of points $\{x_{n}^{j}\}_{j=1}^{\infty}\subset \mathbb{R}^{2}$ for each $n$, which satisfy properties (3.1)-(3.5).
By (1.10), Lemma 1 and Theorem 3, we find
\begin{align}
m^{4} &\leq \limsup_{n\rightarrow \infty}\int
\mathcal{L}(|v_{n}|^{2})|v_{n}|^{2}dx \notag \\
&\leq \sum_{j=1}^{\infty}\int \mathcal{L}(|V^{j}|^{2})|V^{j}|^{2}dx \notag
\\
&\leq C_{opt}\sum_{j=1}^{\infty}\|
V^{j}\|^{2}_{2}\|\nabla V^{j}\|^{2}_{2}\notag
\\
&\leq C_{opt}\Big(\sup_{j\geq 1}
\|V^{j}\|_{2}^{2}\Big)\sum_{j=1}^{\infty}\|\nabla V^{j}\|^{2}_{2}.
\end{align}
Then by (3.5) and (1.10),
\begin{align}
\sum_{j=1}^{\infty}\|\nabla V^{j}\|^{2}_{2} \leq
\limsup_{n\rightarrow \infty}\|\nabla v_{n}\|^{2}_{2}  \leq M^{2}.
\end{align}
Combining (3.7) and (3.8), we have
\begin{align}
\sup_{j\geq 1} \|V^{j}\|_{2}^{2} \geq
\frac{m^{4}}{M^{2}C_{opt}}.
\end{align}
By (3.4) the series $\sum_{j=1}^{\infty}\|V^{j}\|^{2}_{2}$ converges,
$\lim_{j\rightarrow \infty}\|V^{j}\|_{2}=0$, and the supremum in (3.9) is
attained.  That is, there exists some $j_{0}$ such that
\begin{align}
\|V^{j_{0}}\|_{2} \geq \frac{m^{2}}{M\sqrt{C_{opt}}}.
\end{align}
With a change of variables, for any $k\geq j_{0}$
\begin{align}
v_{n}(x+x_{n}^{j_{0}}) = V^{j_{0}}(x)+\sum_{1\leq j \leq k, j\neq
j_{0}}V^{j}(x+x^{j_{0}}_{n}-x^{j}_{n}) + \widetilde{v}_{n}^{k}(x),
\end{align}
where $\widetilde{v}^{k}_{n}=v^{k}_{n}(x+x^{j_{0}}_{n})$.
Then by (3.1), for $j\neq j_{0}$
\begin{align}
V^{j}(\cdot + x^{j_{0}}_{n} - x^{j}_{n})\rightharpoonup 0 \ \
\text{weakly in}\  H^{1}(\mathbb{R}^{2}) \ \text{as}\  n\rightarrow \infty.
\end{align}
Up to a subsequence, we can assume that
\begin{align}\tilde{v}_{n}^{k}\rightharpoonup \tilde{v}^{k} \ \ \text{weakly in}\  H^{1}(\mathbb{R}^{2}) \ \text{as}\  n\rightarrow \infty
\end{align}
for some  $\tilde{v}^{k}\in H^{1}(\mathbb{R}^{2})$, and all $k\geq 0$.
We justify this assumption.  Fixing $k=1, \{\widetilde{v}_{n}^{1}\}_{n=1}^{\infty}$ is bounded in
$H^{1}(\mathbb{R}^{2})$. By extracting a subsequence
$\{\widetilde{v}_{n^{1}_{i}}^{1}\}_{i=1}^{\infty}$, we can ensure
$\widetilde{v}_{n^{1}_{i}}^{1} \rightharpoonup \widetilde{v}^{1}$ as
$i\rightarrow \infty$.  For $k=2,
\{\widetilde{v}_{n^{1}_{i}}^{2}\}_{i=1}^{\infty}$ is bounded in
$H^{1}(\mathbb{R}^{2})$, so once again extracting a subsequence
$\{n^{2}_{i}\}\subset \{n^{1}_{i}\}$, we can ensure
$\widetilde{v}_{n^{2}_{i}}^{2} \rightharpoonup \widetilde{v}^{2}$ as
$i\rightarrow \infty$.  Proceeding this way,
such that $\widetilde{v}_{n^{k}_{i}}^{k}
\rightharpoonup \widetilde{v}^{k}$ as $i\rightarrow \infty$  for each $k\geq1$,
we can
extract the diagonal subsequence
$\{\widetilde{v}_{n^{i}_{i}}^{k}\}_{i=1}^{\infty}$ which satisfies
 $\widetilde{v}_{n^{i}_{i}}^{k}
\rightharpoonup \widetilde{v}^{k}$ as $i\rightarrow \infty$ for all
$k\geq1$, since $\{n^{i}_{i}\}\subset \{n^{k}_{i}\}$ for $i$ sufficiently large.  Rename this subsequence
$\{\widetilde{v}_{n}^{k}\}_{n=1}^{\infty}$, and (3.13) is justified.
\\ \indent
By (3.11), (3.12) and (3.13),
\begin{align*}
v_{n}(x+x_{n}^{j_{0}}) \rightharpoonup V^{j_{0}} + \widetilde{v}^{k}
\ \ \text{weakly in} \ H^{1}(\mathbb{R}^{2}) \ \text{as}\  n\rightarrow \infty.
\end{align*}
The sequence $v_{n}(x+x_{n}^{j_{0}})$ is independent of $k$, and hence the
weak limit $V^{j_{0}} + \widetilde{v}^{k}$ is also independent of $k$.  Therefore, for all $k\geq j_{0}$,
$\widetilde{v}^{k}=\widetilde{v}^{j_{0}}$ for some
$\widetilde{v}^{j_{0}}\in H^{1}(\mathbb{R}^{2})$.  By
lower semi-continuity of the $L^{4}(\mathbb{R}^{2})$ norm in the weak limit,
\begin{align*}
\|\widetilde{v}^{j_{0}}\|_{4} \leq \limsup_{n\rightarrow
\infty}\|\widetilde{v}^{k}_{n}\|_{4}= \limsup_{n\rightarrow
\infty}\|v^{k}_{n}\|_{4}\rightarrow 0 \ \ \text{as}\ k\rightarrow
\infty, \ \ \text{by (3.3)}.
\end{align*}
Thus $\widetilde{v}^{j_{0}}=0$, and,
\begin{align}
v_{n}(x+x_{n}^{j_{0}}) \rightharpoonup V^{j_{0}} \ \ \text{weakly in}\  H^{1}(\mathbb{R}^{2}) \ \text{as}\  n\rightarrow \infty.
\end{align}
By (3.10) and (3.14), the function $V^{j_{0}}$ and the sequence
$\{x_{n}^{j_{0}}\}_{n=1}^{\infty}$ satisfy the claims of Theorem 5.
This completes the proof of Theorem 5 under the assumption that Lemma 1 holds true.
\qed

The proof of Lemma 1 requires two elementary results.
\begin{lemma}
Suppose $\phi,\psi \in H^{1}(\mathbb{R}^{2})$, if
$|x_{n}^{1}-x_{n}^{2}|\rightarrow \infty$ as $n\rightarrow \infty$,
then
\begin{align*}&\|\phi(\cdot-x^{1}_{n})\psi(\cdot-x^{2}_{n})\|_{2}\rightarrow
0 \ \ \text{as} \ n\rightarrow \infty.
\end{align*}
\end{lemma}
\begin{proof}
By Sobolev embedding $\phi,\psi \in L^{4}(\mathbb{R}^{2})$.  Given $\epsilon>0$, we use the density
of $C^{\infty}_{0}(\mathbb{R}^{2})$ in $L^{4}(\mathbb{R}^{2})$ to find $\phi_{0},\psi_{0}\in C^{\infty}_{0}(\mathbb{R}^{2})$ such
that $\|\phi-\phi_{0}\|_{4} < \frac{\epsilon}{2\|\psi\|_{4}}$, and $\|\psi-\psi_{0}\|_{4} < \frac{\epsilon}{2\|\phi_{0}\|_{4}}$.  Applying the triangle and H\"{o}lder inequalities
\begin{align}
\|\phi(\cdot-x_{n}^{1})\psi(\cdot-x_{n}^{2})\|_{2} &=
\|\phi_{0}(\cdot-x_{n}^{1})(\psi(\cdot-x_{n}^{2})-\psi_{0}(\cdot-x_{n}^{2})) \notag \\
& \ \ \ \ \ +
(\phi(\cdot-x_{n}^{1})-\phi_{0}(\cdot-x_{n}^{1}))\psi(\cdot-x_{n}^{2})
+ \phi_{0}(\cdot-x_{n}^{1})\psi_{0}(\cdot-x_{n}^{2})\|_{2} \notag \\
&\leq \|\phi_{0}\|_{4}\|\psi-\psi_{0}\|_{4} +
\|\psi\|_{4}\|\phi-\phi_{0}\|_{4} +
\|\phi_{0}(\cdot-x_{n}^{1})\psi_{0}(\cdot-x_{n}^{2})\|_{2} \notag \\
& < \frac{\epsilon}{2} + \frac{\epsilon}{2} + 0  \\
&= \epsilon. \notag
\end{align}
(3.15) follows from $\|\phi_{0}(\cdot-x_{n}^{1})\psi_{0}(\cdot-x_{n}^{2})\|_{2}=0$ for $n$ sufficiently large, which is due to $|x_{n}^{1}-x_{n}^{2}|\rightarrow \infty$ as $n\rightarrow \infty$, and $\phi_{0},\psi_{0}\in C^{\infty}_{0}(\mathbb{R}^{2})$.
\end{proof}
\begin{lemma}
Suppose $\phi\in L^{2}(\mathbb{R}^{2}), \psi \in
H^{1}(\mathbb{R}^{2})$, if $|x_{n}^{1}-x_{n}^{2}|\rightarrow \infty$
as $n\rightarrow \infty$, then
\begin{align*}\int|\phi(x-x^{1}_{n})||\psi(x-x^{2}_{n})|^{2}dx\rightarrow 0\ \ \text{as} \ n\rightarrow \infty.
\end{align*}
\end{lemma}
\begin{proof}
We choose $\phi_{0},\psi_{0}\in C^{\infty}_{0}(\mathbb{R}^{2})$ such
that $\|\phi-\phi_{0}\|_{2}<\frac{\epsilon}{2\|\psi\|_{4}^{2}}$, and
$\|\psi-\psi_{0}\|_{4}<\frac{\epsilon}{B}$, where $B$ will be
determined afterward. Then,
\begin{align*}
\int|\phi(x-x_{n}^{1})||\psi(x-x_{n}^{2})|^{2}dx &=
\int|(\phi(x-x_{n}^{1})-\phi_{0}(x-x_{n}^{1}))(\psi(x-x_{n}^{2}))^{2} \\
& \ \ \ \ \ \ \ \  + \phi_{0}(x-x_{n}^{1})((\psi(x-x_{n}^{2}))^{2}-(\psi_{0}(x-x_{n}^{2}))^{2})  \\
& \ \ \ \ \ \ \ \  + \phi_{0}(x-x_{n}^{1})(\psi_{0}(x-x_{n}^{2}))^{2}|dx \\
&\leq \|\psi^{2}\|_{2}\|\phi-\phi_{0}\|_{2} + \|\phi_{0}\|_{2}\|\psi^{2}-\psi_{0}^{2}\|_{2}  \\
& \ \ \ \ \ \ \ \  +
\int|\phi_{0}(x-x_{n}^{1})||\psi_{0}(x-x_{n}^{2})|^{2}dx \\
&\leq \|\psi\|_{4}^{2}\|\phi-\phi_{0}\|_{2} + \|\phi_{0}\|_{2}\|\psi + \psi_{0}\|_{4}\|\psi-\psi_{0}\|_{4} \\
& \ \ \ \ \ \ \ \  +
\int|\phi_{0}(x-x_{n}^{1})||\psi_{0}(x-x_{n}^{2})|^{2}dx \\
&\leq
\|\psi\|_{4}^{2}\big(\frac{\epsilon}{2\|\psi\|_{4}^{2}}\big) +
\|\phi_{0}\|_{2}(2\|\psi\|_{4}+\|\psi-\psi_{0}\|_{4})\|\psi-\psi_{0}\|_{4} \\
& \ \ \ \ \ \ \ \  +
\int|\phi_{0}(x-x_{n}^{1})||\psi_{0}(x-x_{n}^{2})|^{2}dx \\
&<
\frac{\epsilon}{2} + \|\phi_{0}\|_{2}\big(2\|\psi\|_{4}+\frac{\epsilon}{B}\big)\frac{\epsilon}{B}
+ 0 \\
&\leq \epsilon
\end{align*}
for $n$ large, as long as
\begin{align*}
\frac{1}{2}B^{2} - 2\|\phi_{0}\|_{2}\|\psi\|_{4}B -
\epsilon\|\phi_{0}\|_{2} \geq 0
\end{align*}
which holds for $B$ sufficiently large.
\end{proof}
\noindent $\textit{Proof of Lemma 1.}$ Observe that the infinite sum on the right-hand side of (3.6) converges
by (3.7), (3.8) and boundedness of $\{v_{n}\}\subset H^{1}(\mathbb{R}^{2})$.  It therefore suffices to show that $\forall\,\epsilon>0$, $\exists \,K(\epsilon)>0$ such that for $k\geq K(\epsilon)$,
\begin{align*}
\limsup_{n\rightarrow \infty}\int
\mathcal{L}(|v_{n}|^{2})|v_{n}|^{2}dx \leq \sum_{j=1}^{k}\int
\mathcal{L}(|V^{j}|^{2})|V^{j}|^{2}dx + \epsilon.
\end{align*}
Letting $W^{k}_{n}:=\sum_{j=1}^{k}V^{j}(x-x^{j}_{n})$, we compute
\begin{align}
\int
\mathcal{L}(|v_{n}|^{2})|v_{n}|^{2}dx &= \int
\mathcal{L}(|W^{k}_{n}+v^{k}_{n}|^{2})|W^{k}_{n}+v^{k}_{n}|^{2}dx \notag \\
&= \int
\mathcal{L}(|W^{k}_{n}|^{2})|W^{k}_{n}|^{2}dx + R_{n}^{k}
\end{align}
where the remainder term is
\begin{align*}
R_{n}^{k} = \int
\mathcal{L}(2Re(W^{k}_{n}\overline{v^{k}_{n}})+|v^{k}_{n}|^{2})|v_{n}|^{2}dx
+ \int
\mathcal{L}(|W^{k}_{n}|^{2})(2Re(W^{k}_{n}\overline{v^{k}_{n}})+|v^{k}_{n}|^{2})dx.
\end{align*}
We also take
\begin{align}
\int
\mathcal{L}(|W^{k}_{n}|^{2})|W^{k}_{n}|^{2}dx = \sum_{j=1}^{k}\int
\mathcal{L}(|V^{j}|^{2})|V^{j}|^{2}dx + C_{n}^{k}
\end{align}
where the mixed cross term is
\begin{align*}
C_{n}^{k} = \sum_{\substack{1\leq i_{1},i_{2},i_{3},i_{4}\leq k \\
\ i_{m}\neq i_{j}\  \\ \text{for some}\ m\neq j}}\int
\mathcal{L}(V_{i_{1}}(x-x_{n}^{i_{1}})\overline{V}_{i_{2}}(x-x_{n}^{i_{2}}))V_{i_{3}}(x-x_{n}^{i_{3}})\overline{V}_{i_{4}}(x-x_{n}^{i_{4}})dx.
\end{align*}
We \textit{claim} that given $\epsilon>0, \exists\,K(\epsilon)>0$ such that
for $k\geq K(\epsilon)$
\begin{align}
\limsup_{n\rightarrow \infty}|R_{n}^{k}|\leq \epsilon \\
\limsup_{n\rightarrow \infty}|C_{n}^{k}|=0.
\end{align}
By Plancherel's theorem and the definition of the operator $\mathcal{B}$,
\begin{align}
\|\mathcal{L}(f)\|_{2}=\|\nu f+\gamma \mathcal{B}(f)\|_{2}\leq \|f\|_{2}+
\gamma\|\mathcal{B}(f)\|_{2} \notag &= \|f\|_{2}+
\gamma\|\widehat{\mathcal{B}(f)}\|_{2} \notag
\\ &= \|f\|_{2}+
\gamma\big(\int|\frac{\xi_{1}^{2}}{|\xi|^{2}}\hat{f}(\xi)|^{2}d\xi
\big)^{1/2} \notag \\ &\leq \|f\|_{2}+ \gamma\|\hat{f}\|_{2} \notag \\
&=(1+\gamma)\|f\|_{2}.
\end{align}
We then compute, with repeated use of H\"{o}lder and triangle inequalities,
\begin{align}
|R_{n}^{k}| &\leq \int
|\mathcal{L}(2Re(W^{k}_{n}\overline{v^{k}_{n}})+|v^{k}_{n}|^{2})|v_{n}|^{2}
\notag\\
&\ \ \ \ \ \ \ \ \ \ \ \ \ \ \ \ \ +
\mathcal{L}(|W^{k}_{n}|^{2})(2Re(W^{k}_{n}\overline{v^{k}_{n}})+|v^{k}_{n}|^{2})|dx
\notag\\
&\leq
\Big(\|\mathcal{L}(2Re(W^{k}_{n}\overline{v^{k}_{n}})+|v^{k}_{n}|^{2})\|_{2}\||v_{n}|^{2}\|_{2}
\notag\\
&\ \ \ \ \ \ \ \ \ \ \ \ \ \ \ \ \ +
\|\mathcal{L}(|W^{k}_{n}|^{2})\|_{2}\|2Re(W^{k}_{n}\overline{v^{k}_{n}})+|v^{k}_{n}|^{2}\|_{2}\Big)
\notag\\
&\leq
(1+\gamma)\|2Re(W^{k}_{n}\overline{v^{k}_{n}})+|v^{k}_{n}|^{2}\|_{2}\Big(\||v_{n}|^{2}\|_{2}
+ \||W^{k}_{n}|^{2}\|_{2}\Big) \ \ \ \text{by (3.20)},
\notag\\
&\leq
(1+\gamma)\Big(2\|W^{k}_{n}v^{k}_{n}\|_{2}+\|v^{k}_{n}\|^{2}_{4}\Big)\Big(\|v_{n}\|^{2}_{4}
+ \|W^{k}_{n}\|^{2}_{4}\Big)
\notag\\
&\leq
(1+\gamma)\|v^{k}_{n}\|_{4}\Big(2\|W^{k}_{n}\|_{4}+\|v^{k}_{n}\|_{4}\Big)\Big(\|v_{n}\|^{2}_{4}
+ \|W^{k}_{n}\|^{2}_{4}\Big)
\notag\\
&\leq
(1+\gamma)\|v^{k}_{n}\|_{4}\Big(2\|v_{n}\|_{4}+3\|v^{k}_{n}\|_{4}\Big)\Big(2\|v_{n}\|^{2}_{4}
+ 2\|v_{n}\|_{4}\|v^{k}_{n}\|_{4} + \|v^{k}_{n}\|_{4}^{2}\Big)
\\ &\leq (1+\gamma)\|v^{k}_{n}\|_{4}\Big(C_{1}+3\|v^{k}_{n}\|_{4}\Big)\Big(C_{2}
+ C_{3}\|v^{k}_{n}\|_{4} + \|v^{k}_{n}\|_{4}^{2}\Big),
\end{align}
where $C_{1},C_{2},C_{3}>0$ are constants.  Here we have obtained (3.21) by writing $W^{k}_{n}=v_{n}-v_{n}^{k}$, and (3.22) follows from boundedness of $\{v_{n}\}\subset H^{1}(\mathbb{R}^{2})$, since $\|v_{n}\|_{4}^{2}\leq
C\|v_{n}\|_{2}\|\nabla v_{n}\|_{2}<\widetilde{C}$, independent of
$n$.  By (3.3)
\begin{align*}
\limsup_{n\rightarrow \infty}|R_{n}^{k}|
&\leq\limsup_{n\rightarrow
\infty}(1+\gamma)\|v^{k}_{n}\|_{4}\Big(C_{1}+3\limsup_{n\rightarrow
\infty}\|v^{k}_{n}\|_{4}\Big) \\
&\ \ \ \ \ \ \ \ \ \cdot\Big(C_{2} + C_{3}\limsup_{n\rightarrow
\infty}\|v^{k}_{n}\|_{4} + \big(\limsup_{n\rightarrow
\infty}\|v^{k}_{n}\|_{4}\big)^{2}\Big) \\
&\leq \epsilon,
\end{align*}
for $k\geq K(\epsilon)$ sufficiently large.  Having
justified $(3.18)$, it remains to justify $(3.19)$.  For
$k\geq K(\epsilon)$
\begin{align*}
|C_{n}^{k}| &= |\sum_{\substack{1\leq i_{1},i_{2},i_{3},i_{4}\leq k \\
\ i_{m}\neq i_{j}\  \\ \text{for some}\ m\neq j}}\int
\mathcal{L}(V_{i_{1}}(x-x_{n}^{i_{1}})\overline{V}_{i_{2}}(x-x_{n}^{i_{2}}))V_{i_{3}}(x-x_{n}^{i_{3}})\overline{V}_{i_{4}}(x-x_{n}^{i_{4}})dx|
\\
&\leq \sum_{\substack{1\leq i_{1},i_{2},i_{3},i_{4}\leq k \\
\ i_{m}\neq i_{j}\  \\ \text{for some}\ m\neq j}}\int
|\mathcal{L}(V_{i_{1}}(x-x_{n}^{i_{1}})\overline{V}_{i_{2}}(x-x_{n}^{i_{2}}))V_{i_{3}}(x-x_{n}^{i_{3}})\overline{V}_{i_{4}}(x-x_{n}^{i_{4}})|dx.
\end{align*}
We organize the terms of this sum as follows
\begin{enumerate}
\renewcommand{\labelenumi}{(\alph{enumi})}
\renewcommand{\labelenumii}{(\alph{enumii})}
\renewcommand{\labelenumiii}{(\alph{enumiii})}
\item $i_{1}\neq i_{2}$
\item $i_{1}=i_{2},i_{3}\neq i_{4}$
\item $i_{1}=i_{2}\neq i_{3}=i_{4}$
\end{enumerate}
In case (a),
\begin{align*}
\limsup_{n\rightarrow \infty}&\int
|\mathcal{L}(V_{i_{1}}(x-x_{n}^{i_{1}})\overline{V}_{i_{2}}(x-x_{n}^{i_{2}}))V_{i_{3}}(x-x_{n}^{i_{3}})\overline{V}_{i_{4}}(x-x_{n}^{i_{4}})|dx
\\ &\leq
\limsup_{n\rightarrow
\infty}\|\mathcal{L}(V_{i_{1}}(\cdot-x_{n}^{i_{1}})\overline{V}_{i_{2}}(\cdot-x_{n}^{i_{2}}))\|_{2}
\|V_{i_{3}}(\cdot-x_{n}^{i_{3}})\overline{V}_{i_{4}}(\cdot-x_{n}^{i_{4}})\|_{2}
\\ &\leq
\limsup_{n\rightarrow
\infty}(1+\gamma)\|V_{i_{1}}(\cdot-x_{n}^{i_{1}})\overline{V}_{i_{2}}(\cdot-x_{n}^{i_{2}})\|_{2}
\|V_{i_{3}}\|_{4}\|V_{i_{4}}\|_{4}\ \ \ \ \text{by (3.20)}
\\
&= 0
\end{align*}
by Lemma 2.  Similarly, in case (b), by Lemma 2,
\begin{align*}
\limsup_{n\rightarrow \infty}&\int
|\mathcal{L}(V_{i_{1}}(x-x_{n}^{i_{1}})\overline{V}_{i_{2}}(x-x_{n}^{i_{2}}))V_{i_{3}}(x-x_{n}^{i_{3}})\overline{V}_{i_{4}}(x-x_{n}^{i_{4}})|dx
\\ &\leq
\limsup_{n\rightarrow
\infty}(1+\gamma)\|V_{i_{1}}\|_{4}^{2}\|V_{i_{3}}(\cdot-x_{n}^{i_{3}})\overline{V}_{i_{4}}(\cdot-x_{n}^{i_{4}})\|_{2}
\\
&= 0 .
\end{align*}
In case (c), $V_{i_{1}}\in H^{1}(\mathbb{R}^{2}) \subset L^{4}(\mathbb{R}^{2})$, and thus $|V_{i_{1}}|^{2}\in L^{2}(\mathbb{R}^{2})$.  By (3.20),
$\mathcal{L}(|V_{i_{1}}|^{2}) \in L^{2}(\mathbb{R}^{2})$.  Applying Lemma 3,
\begin{align*}\limsup_{n\rightarrow \infty}\int
|\mathcal{L}(|V_{i_{1}}(x-x_{n}^{i_{1}})|^{2})||V_{i_{3}}(x-x_{n}^{i_{3}})|^{2}dx
&= \limsup_{n\rightarrow \infty}\int
|\mathcal{L}(|V_{i_{1}}|^{2})(x-x_{n}^{i_{1}})||V_{i_{3}}(x-x_{n}^{i_{3}})|^{2}dx \\
&= 0.
\end{align*}
Therefore
\begin{align*}\limsup_{n\rightarrow \infty}|C_{n}^{k}|
&\leq \sum_{\substack{1\leq i_{1},i_{2},i_{3},i_{4}\leq k \\
\ i_{m}\neq i_{j}\  \\ \text{for some}\ m\neq
j}}\limsup_{n\rightarrow \infty}\int
|\mathcal{L}(V_{i_{1}}(x-x_{n}^{i_{1}})\overline{V}_{i_{2}}(x-x_{n}^{i_{2}}))V_{i_{3}}(x-x_{n}^{i_{3}})\overline{V}_{i_{4}}(x-x_{n}^{i_{4}})|dx
\\
&=0
\end{align*}
for $k\geq K(\epsilon)$, and (3.19) holds.  Combining (3.16),(3.17),(3.18) and (3.19)
\begin{align*}
\limsup_{n\rightarrow \infty}\int \mathcal{L}(|v_{n}|^{2})|v_{n}|^{2}dx
&\leq
\sum_{j=1}^{k}\int
\mathcal{L}(|V^{j}|^{2})|V^{j}|^{2}dx  + \limsup_{n\rightarrow \infty}|R_{n}^{k}| +  \limsup_{n\rightarrow \infty}|C_{n}^{k}| \\
&\leq \sum_{j=1}^{k}\int
\mathcal{L}(|V^{j}|^{2})|V^{j}|^{2}dx + \epsilon
\end{align*} for $k\geq K(\epsilon)$.  This completes the proof of Lemma 1.
\qed

\section{Mass Concentration for $L^{2}$-solutions}
For the proof of Theorem 7 we require two lemmas from \cite{B1}.

\begin{lemma}[Squares Lemma] Given $f\in L^{2}(\mathbb{R}^{2})$, $\epsilon>0$, $\exists\, (f_{r})_{1\leq r < R(\epsilon)}\subset L^{2}(\mathbb{R}^{2})$ such that
\begin{enumerate}
\renewcommand{\labelenumi}{(\alph{enumi})}
\renewcommand{\labelenumii}{(\alph{enumii})}
\renewcommand{\labelenumiii}{(\alph{enumiii})}
\renewcommand{\labelenumiv}{(\alph{enumiv})}
\item $\emph{supp}\,\hat{f}_{r} \subset \tau_{r} \subset \subset \mathbb{R}^{2}$, where $\tau_{r}$ is a square of side  $\ell_{r}$, centre $\xi_{r}$. \\
\item  $|\hat{f}_{r}| \lesssim_{\epsilon} \frac{1}{\ell_{r}}$. \\
\item  $\|f_{r}\|_{2} \geq \epsilon'(\epsilon) >0$.  \\
\item  $\|e^{it\Delta}f- \Sigma_{1\leq r < R(\epsilon)}e^{it\Delta}f_{r}\|_{L^{4}_{x,t}} < \epsilon$.
\end{enumerate}
\end{lemma}
\begin{proof} (Outline, a complete proof is found in \cite[p. 255]{B1}, section 2.)
The properties (a)-(d) are scale invariant, by adjusting $\ell_{r}$ as needed.  We can therefore assume $~\mbox{supp}\, \hat{f} \subset B(0,1)$.  For each $j=1,2,3...$, let $\delta_{j}=2^{-j}$, and take $C_{j}$ to be a grid of $\delta_{j}\times \delta_{j}$ squares partitioning $B(0,1)$.  Fixing $12/7<p<2$, the key input for this proof is the following Strichartz refinement estimate from \cite{MVV}
\begin{align*}
\|e^{it\Delta}f\|^{4}_{L^{4}_{x,t}} &\leq C\Big[\sum_{j=1}^{\infty}\sum_{\tau\in C_{j}}\delta_{j}^{4}\Big(\frac{1}{|\tau|}\int_{\tau}|\hat{f}|^{p}\Big)^{4/p}\Big].
\end{align*}
This leads to
\begin{align}
\|e^{it\Delta}f\|^{4}_{L^{4}_{x,t}}
&\leq  C\Big[\sum_{j=1}^{\infty}\sum_{\tau\in C_{j}}
\frac{1}{\delta_{j}^{2(2-p)}}\Big(\int_{\tau}|\hat{f}|^{p}\Big)^{2}\Big]\max_{\substack{j\in \mathbb{Z^{+}} \\ \tau \in C_{j}}} \Big(\delta_{j}^{p-2}\int_{\tau}|\hat{f}|^{p}\Big)^{\frac{2}{p}(2-p)}.
\end{align}
We bound the first factor in (4.1) by decomposing $(B(0,1))^{2}$ into disjoint sets of the form
\begin{align*}
\Lambda_{j}=\{(x,y)\in B(0,1)\times B(0,1):\delta_{j+1}< |x-y|\leq \delta_{j}\},
\end{align*}
further decomposing each $\Lambda_{j}$ into 4-dimensional cubes $\sigma$ of sidelength $\delta_{j}$, and considering
\begin{align}
\iint_{(B(0,1))^{2}}\frac{|\hat{f}(x)|^{p}|\hat{f}(y)|^{p}}{|x-y|^{2(2-p)}}dxdy &= \sum_{j=1}^{\infty}
\sum_{\sigma\in \Lambda_{j}}\iint_{\sigma}\frac{|\hat{f}(x)|^{p}|\hat{f}(y)|^{p}}{|x-y|^{2(2-p)}}dxdy \notag \\ &\geq \sum_{j=1}^{\infty}
\sum_{\sigma\in \Lambda_{j}}\frac{1}{\delta_{j}^{2(2-p)}}\iint_{\sigma}|\hat{f}(x)|^{p}|\hat{f}(y)|^{p}dxdy  \notag  \\
&= \sum_{j=1}^{\infty} \sum_{\tau \in C_{j}}\frac{1}{\delta_{j}^{2(2-p)}}\Big(\int_{\tau}|\hat{f}|^{p}\Big)^{2}.
\end{align}
By Hardy-Littlewood-Sobolev, we have
\begin{align}
\iint_{(B(0,1))^{2}}\frac{|\hat{f}(x)|^{p}|\hat{f}(y)|^{p}}{|x-y|^{2(2-p)}}dxdy \leq C\|f\|_{2}^{2p} = C.
\end{align}
\noindent If $\|e^{it\Delta}f\|_{4}<\epsilon$, there is nothing to prove.  If $\|e^{it\Delta}f\|_{4}\geq \epsilon$, then by (4.1),(4.2) and (4.3)
\begin{align}
\frac{1}{\delta_{j}^{2-p}}\int_{\tau}|\hat{f}|^{p} > C\epsilon^{\frac{2p}{2-p}}
\end{align}
for some $\delta_{j}$-square $\tau$.  With two applications of H\"{o}lder's inequality, (4.4) gives
\begin{align*}
C\epsilon^{\frac{2p}{2-p}} &= \frac{1}{\delta_{j}^{2-p}}\int_{\tau\cap \{|\hat{f}|<M\}}|\hat{f}|^{p} + \frac{1}{\delta_{j}^{2-p}}\int_{\tau\cap \{|\hat{f}|\geq M\}}|\hat{f}|^{p} \\
&\leq \big(\int_{\tau\cap \{|\hat{f}|<M\}}|\hat{f}|^{2}\big)^{p/2} + \frac{1}{(M\delta_{j})^{2-p}}\|f\|_{2}^{2} \\
&\lesssim \big(\int_{\tau\cap \{|\hat{f}|<M\}}|\hat{f}|^{2}\big)^{p/2} + \frac{C}{2}\epsilon^{\frac{2p}{2-p}}, \ \ \ \text{by choosing} \ \ M \sim \frac{1}{\epsilon^{\frac{2p}{(2-p)^{2}}}\delta_{j}}.
\end{align*}
That is,
\begin{align}
\int_{\tau\cap \{|\hat{f}|<M\}}|\hat{f}|^{2} > \frac{C}{2}\epsilon^{\frac{4}{2-p}}.
\end{align}
The function $f_{1}$ defined by $\hat{f_{1}}=\hat{f}\chi_{\tau\cap\{|\hat{f}|<M\}}$ will satisfy properties (a)-(c) with $\ell_{1}=\delta_{j}$, where $\delta_{j}$ is the sidelength of the square $\tau$.  Replacing $f$ by $f^{1}=f-f_{1}$, $\hat{f^{1}}$ and $\hat{f_{1}}$ are orthogonal in $L^{2}(\mathbb{R}^{2})$, and thus
\begin{align}
\|\hat{f^{1}}\|_{2}^{2}&= \|\hat{f}\|_{2}^{2} - \|\hat{f_{1}}\|_{2}^{2}  \notag \\
&\leq \|f\|_{2}^{2}- \frac{C}{2}\epsilon^{\frac{4}{2-p}}.
\end{align}
If $\|e^{it\Delta}f^{1}\|_{L^{4}_{x,t}} > \epsilon$, repeat this procedure.  By (4.6), this process will terminate in finitely many steps $R=R(\epsilon)$, producing a sequence of functions $(f_{r})_{1\leq r\leq R(\epsilon)}$ satisfying properties (a)-(d).
\end{proof}
\begin{lemma}[Tubes Lemma] Given $g\in L^{2}(\mathbb{R}^{2})$ such that
\begin{enumerate}
\renewcommand{\labelenumi}{(\alph{enumi})}
\renewcommand{\labelenumii}{(\alph{enumii})}
\item $\,\emph{supp}\,\hat{g} \subset \tau \subset \subset \mathbb{R}^{2}$ where $\tau$ is a square of side  $\ell$, centre $\xi_{0}$.
\item  $|\hat{g}| \leq \frac{1}{\ell}$. \\
\end{enumerate}
$\forall \epsilon >0, \  \exists$ tubes  $(Q_{s})_{1\leq s < S(\epsilon)}$ of the form
$
Q_{s} = \{ (t,x)\in \mathbb{R}^{3}: x+ 2t\xi_{0} \in \tau_{s},t\in J_{s}\}$ such that $\text{side}(\tau_{s})=\ell^{-1}, \ |J_{s}|=\ell^{-2}$, and  \begin{align*}  \big(\int_{\mathbb{R}^{3}\setminus \{\cup_{1\leq s < S(\epsilon)}Q_{s}\}}|e^{it\Delta}g|^{4}dxdt\big)^{1/4}<\epsilon.
\end{align*}
\end{lemma}
\begin{proof} (Outline, a complete proof is found in \cite[p. 257]{B1}, section 3.)
Let $g'$ be the function defined by
\begin{align}
\hat{g'}(\xi)=\ell\hat{g}(\xi_{0}+\ell\xi), \ \ \text{and take} \ t'=\ell^{2}t.
\end{align}
From the assumptions (a),(b), this implies
\begin{align}
\mbox{supp}\,\hat{g'} \subset B(0,1), \ \ \text{and} \ |\hat{g'}|<1.
\end{align}
The key ingredient of this proof is the following estimate from \cite{B2}: $\exists \, q^{*}<4$ such that $\forall q\in(q^{*},4]$
\begin{align*}
\|\int_{B(0,1)}F(\xi)e^{i(x\cdot\xi + t|\xi|^{2})}d\xi\|_{L^{q}_{\mathbb{R}^{3}}} \leq C \|F\|_{L^{\infty}_{\mathbb{R}^{2}}}.
\end{align*}
Fix a $q\in(q^{*},4]$, then this gives
\begin{align*}
\|e^{it'\Delta}g'\|_{L^{q}_{\mathbb{R}^{3}}} \leq C.
\end{align*}
Consider, for $\lambda>0$, that
\begin{align*}
\frac{1}{\lambda^{4-q}}\Big( \int_{\{|e^{it'\Delta}g'|<\lambda\}}|e^{it'\Delta}g'|^{4}dxdt\Big) <
 \int_{\{|e^{it'\Delta}g'|<\lambda\}}|e^{it'\Delta}g'|^{q}dxdt < C.
\end{align*}
Choosing $\lambda=\lambda(\epsilon)\sim \epsilon^{\frac{1}{1-q/4}}$, we find
\begin{align*}
\int_{\{|e^{it'\Delta}g'|<\lambda\}}|e^{it'\Delta}g'|^{4}dxdt < C\lambda^{4-q} < \epsilon^{4}.
\end{align*}
Letting $\mu$ denote Lebesgue measure in $\mathbb{R}^{3}$, then
\begin{align*}
\lambda^{q}\mu(\{|e^{it'\Delta}g'|>\lambda\}) \leq \int_{\{|e^{it'\Delta}g'|>\lambda\}}|e^{it'\Delta}g'|^{q}dxdt \leq C,
\end{align*}
which gives
\begin{align*}
\mu(\{|e^{it'\Delta}g'|>\lambda\}) \leq  \frac{C}{\lambda^{q}} < \infty.
\end{align*}
By (4.8), $e^{it'\Delta}g'(x')$ is Lipschitz in $t',x'$.  That is,
\begin{align}
|(e^{it'\Delta}g')(x')-(e^{it''\Delta}g')(x'')| \leq C(|x'-x''|+|t'-t''|).
\end{align}
Cover $\{|e^{it\Delta}g'|>\lambda\}$ with finitely many disjoint unit cubes $B_{s}\subset\mathbb{R}^{3}$.  Let $S$ denote the total number of cubes.  By (4.9) each cube $B_{s}$ contains a subcube $D_{s}$ of side $\sim \lambda$, such that for all $x\in D_{s}$, $|e^{it'\Delta}g'(x)|>\frac{\lambda}{2}$.
Therefore
\begin{align*}
C\geq \int_{\mathbb{R}^{3}}|e^{it\Delta}g|^{4}dxdt \geq \int_{\cup_{s}D_{s}} |e^{it\Delta}g|^{4}dxdt \gtrsim (side(D_{s}))^{3}\big(\frac{\lambda}{2}\big)^{4}S \sim \lambda^{7}S \sim \epsilon^{\frac{7}{1-q/4}}S.
\end{align*}
Thus $\{|e^{it'\Delta}g'|>\lambda\}$ can be covered by at most $S=S(\epsilon)\lesssim \epsilon^{\frac{-7}{1-q/4}}$ disjoint unit cubes $B_{s}\subset \mathbb{R}^{3}$.  These cubes satisfy
\begin{align*}
\int_{\mathbb{R}^{3}\setminus \cup_{1\leq s \leq S(\epsilon)}B_{s}}|e^{it'\Delta}g'|^{4}dxdt \leq \int_{\{|e^{it'\Delta}g'|<\lambda\}}|e^{it'\Delta}g'|^{4}dxdt < \epsilon^{4}.
\end{align*}
Undoing the scaling (4.7), the unit cubes $(B_{s})_{1\leq s \leq S(\epsilon)}$ become tubes $(Q_{s})_{1\leq s \leq S(\epsilon)}$ which satisfy the claims of Lemma 5.
\end{proof}
\begin{remark}An alternate proof of Lemma 5, which generalizes to higher dimensions, is found in \cite{BV}. See also \cite{CK}.\end{remark}

\noindent \textit{Proof of Theorem 7.}  Suppose $L^{2}(\mathbb{R}^{2})\ni u_{0}\mapsto u(t)$ is a maximal in time solution to (1.3) such that $T^{*}< \infty$.  From part (a) of Theorem 1, this necessitates
\begin{align*}
\|u\|_{L^{4}_{[0,T^{*}]\times \mathbb{R}^{2}}}= \infty.
\end{align*}
We partition $[0,T^{*})=\cup_{j\geq1}I_{j}$ into disjoint intervals $I_{j}=[t_{j},t_{j+1})$, such that
\begin{align}
\|u\|_{L^{4}_{I_{j}\times \mathbb{R}^{2}}}= \lambda
\end{align}
for all $j$, for some fixed $\lambda\ll 1$.  The integral formulation of (1.3) on the interval $I_{j}$ is given by
\begin{align*}
u(t)=e^{i(t-t_{j})\Delta}u(t_{j}) + i\int_{t_{j}}^{t}e^{i(t-s)\Delta}(u(s)\mathcal{L}(|u(s)|^{2}))ds.
\end{align*}
This implies
\begin{align}
\|u(t)-e^{i(t-t_{j})\Delta}u(t_{j})\|_{L^{4}_{I_{j}\times \mathbb{R}^{2}}} & = \|\int_{t_{j}}^{t}e^{i(t-s)\Delta}(u(s)\mathcal{L}(|u(s)|^{2}))ds\|_{L^{4}_{I_{j}\times \mathbb{R}^{2}}} \notag\\
&\leq \| \int_{t_{j}}^{t}\|e^{i(t-s)\Delta}(u(s)\mathcal{L}(|u(s)|^{2}))\|_{L^{4}_{\mathbb{R}^{2}}}ds\|_{L^{4}_{I_{j}}} \notag\\
& \lesssim \|\int_{t_{j}}^{t}\frac{1}{|t-s|^{1/2}}\|u(s)\mathcal{L}(|u(s)|^{2})\|_{L^{4/3}_{\mathbb{R}^{2}}}ds\|_{L^{4}_{I_{j}}} \\ &\lesssim \|u\mathcal{L}(|u|^{2})\|_{L^{4/3}_{I_{j}\times \mathbb{R}^{2}}} \\
&= \|\big(\int_{\mathbb{R}^{2}}|u|^{4/3}|\mathcal{L}(|u|^{2})|^{4/3}dx\big)^{3/4}\|_{L^{4/3}_{I_{j}}}.\notag
\end{align}
Here (4.11) follows from the operator bound $\|e^{is\Delta}\|_{L^{4/3}\rightarrow L^{4}}\lesssim \frac{1}{|s|^{1/2}}$, and (4.12) from $\|\int\frac{1}{|t-s|^{1/2}}\psi(s)ds \|_{4} \lesssim \|\psi\|_{4/3}$.
By H\"{o}lder's inequality, and (3.20),
\begin{align*}
\int_{\mathbb{R}^{2}}|u|^{4/3}|\mathcal{L}(|u|^{2})|^{4/3}dx & \leq  \big(\int_{\mathbb{R}^{2}}|u|^{4}dx\big)^{1/3}\big(\int_{\mathbb{R}^{2}}|\mathcal{L}(|u|^{2})|^{2}dx\big)^{2/3} \\
& \leq \big(\int_{\mathbb{R}^{2}}|u|^{4}dx\big)^{1/3}\big((1+\gamma)^{2}\int_{\mathbb{R}^{2}}|u|^{4}dx\big)^{2/3} \\
& = (1+ \gamma)^{4/3}\|u\|_{L^{4}_{\mathbb{R}^{2}}}^{4}.
\end{align*}
This gives
\begin{align}
\|u(t)-e^{i(t-t_{j})\Delta}u(t_{j})\|_{L^{4}_{I_{j}\times \mathbb{R}^{2}}} &\lesssim (1+ \gamma)\|u\|^{3}_{L^{4}_{I_{j}\times \mathbb{R}^{2}}} \notag \\
& \lesssim \lambda^{3},
\end{align}
and then
\begin{align}
\|e^{i(t-t_{j})\Delta}u(t_{j})\|_{L^{4}_{I_{j}\times \mathbb{R}^{2}}} &\leq \|u(t)\|_{L^{4}_{I_{j}\times \mathbb{R}^{2}}} + \|e^{i(t-t_{j})\Delta}u(t_{j})-u(t)\|_{L^{4}_{I_{j}\times \mathbb{R}^{2}}} \notag \\ &\lesssim \lambda + \lambda^{3} \notag \\
&\lesssim \lambda.
\end{align}
Justifying (4.13) and (4.14) as above is the only place where structure specific to (1.1) will be invoked.
The remainder of our proof of Theorem 7 will mimic the proof from \cite{B1} exactly.  By (4.10), (4.13), (4.14), and H\"{o}lder's inequality, we find
\begin{align*}
\lambda^{4}&= \int_{I_{j}}\int_{\mathbb{R}^{2}}|u(t)|^{4}dxdt \\ &= \int_{I_{j}}\int_{\mathbb{R}^{2}}u(t)(e^{i(t-t_{j})\Delta}u(t_{j}) + (u(t)-e^{i(t-t_{j})\Delta}u(t_{j}))) \\ &\ \ \ \ \ \ \ \ \  \ \ \ \ \ \ \ \ \cdot\overline{(e^{i(t-t_{j})\Delta}u(t_{j}) + (u(t)-e^{i(t-t_{j})\Delta}u(t_{j})))}\,\overline{(e^{i(t-t_{j})\Delta}u(t_{j}) + (u(t)-e^{i(t-t_{j})\Delta}u(t_{j})))}dxdt \\
&= \int_{I_{j}}\int_{\mathbb{R}^{2}}u(t)(e^{i(t-t_{j})\Delta}u(t_{j}))(\overline{e^{i(t-t_{j})\Delta}u(t_{j})})^{2}dxdt + O(\lambda^{6}).
\end{align*}
Applying Lemma 4 to $f:=u(t_{j})$, with $\epsilon=\lambda^{2}$, $\exists$ functions $(f_{r})_{1\leq r < R(\lambda^{2})}\subset L^{2}(\mathbb{R}^{2})$ satisfying properties $(a)$-$(d)$.  Property (d), (4.10), (4.14) and H\"{o}lder's inequality imply that
\begin{align*}
\lambda^{4} &= \sum_{r_{1},r_{2},r_{3}<R(\lambda^{2})}\int_{I_{j}}\int_{\mathbb{R}^{2}}u(t)(e^{i(t-t_{j})\Delta}f_{r_{1}})(\overline{e^{i(t-t_{j})\Delta}f_{r_{2}}})(\overline{e^{i(t-t_{j})\Delta}f_{r_{3}}})dxdt + O(\lambda^{5}).
\end{align*}
Therefore there is a choice of $r_{1},r_{2},r_{3}<R(\lambda^{2})$ such that
\begin{align*}
\int_{I_{j}}\int_{\mathbb{R}^{2}}u(t)(e^{i(t-t_{j})\Delta}f_{r_{1}})(\overline{e^{i(t-t_{j})\Delta}f_{r_{2}}})(\overline{e^{i(t-t_{j})\Delta}f_{r_{3}}})dxdt > \frac{\lambda^{4}}{(R(\lambda^{2}))^{3}} =: \eta >0.
\end{align*}
Here $\mbox{supp}\,(\mathcal{F}({e^{i(t-t_{j})\Delta}f_{r_{i}}}))= \mbox{supp}\,(\hat{f}_{r_{i}})\subset \tau_{r_{i}}$, a square of side $\ell_{r_{i}}>0$.  Assume $\ell_{r_{1}}\geq \ell_{r_{2}}\geq \ell_{r_{3}}$.  Letting $\psi_{i}:=e^{i(t-t_{j})\Delta}f_{r_{i}}$, by Plancherel's theorem we can write
\begin{align*}
\int_{\mathbb{R}^{2}} u(x)\psi_{1}(x)\overline{\psi_{2}(x)}\overline{\psi_{3}(x)}dx = \iiint_{(\mathbb{R}^{2})^{3}}\hat{u}(\xi)\hat{\psi}_{1}(\xi_{1}-\xi)\hat{\overline{\psi}}_{2}(\xi_{2})\hat{\overline{\psi}}_{3}(\xi_{1}-\xi_{2})d\xi_{2}d\xi_{1}d\xi.
\end{align*}
The product $\hat{u}(\xi)\hat{\psi}_{1}(\xi_{1}-\xi)\hat{\overline{\psi}}_{2}(\xi_{2})\hat{\overline{\psi}}_{3}(\xi_{1}-\xi_{2})$ is non-zero only when
$\xi_{1}-\xi \in \tau_{r_{1}}, \xi_{2} \in \tau_{r_{2}},$ and $\xi_{1}-\xi_{2} \in \tau_{r_{3}}.$
This implies that $\xi \in \tau, \ \text{a square of sidelength} \ \ell := 3\ell_{r_{1}}.$
Let $P_{\tau}$ be the Fourier restriction operator defined by $\widehat{P_{\tau}f}=\chi_{\tau}\hat{f}$, where $\chi_{\tau}$ is the characteristic function of the square $\tau$.  We find
\begin{align*}
\int_{\mathbb{R}^{2}} u(x)\psi_{1}(x)\overline{\psi_{2}(x)}\overline{\psi_{3}(x)}dx &= \iiint_{(\mathbb{R}^{2})^{3}}\hat{u}(\xi)\hat{\psi}_{1}(\xi_{1}-\xi)\hat{\overline{\psi}}_{2}(\xi_{2})\hat{\overline{\psi}}_{3}(\xi_{1}-\xi_{2})d\xi_{2}d\xi_{1}d\xi
\\ &=\iiint_{(\mathbb{R}^{2})^{3}}\widehat{P_{\tau}u}(\xi)\hat{\psi}_{1}(\xi_{1}-\xi)\hat{\overline{\psi}}_{2}(\xi_{2})\hat{\overline{\psi}}_{3}(\xi_{1}-\xi_{2})d\xi_{2}d\xi_{1}d\xi
\\
&=\int_{\mathbb{R}^{2}} P_{\tau}u(x)\psi_{1}(x)\overline{\psi_{2}(x)}\overline{\psi_{3}(x)}dx.
\end{align*}
This leads to
\begin{align*}
\eta &< \iint_{I_{j}\times\mathbb{R}^{2}} u(t)(e^{i(t-t_{j})\Delta}f_{r_{1}})(\overline{e^{i(t-t_{j})\Delta}f_{r_{2}}})(\overline{e^{i(t-t_{j})\Delta}f_{r_{3}}})dxdt  \\
&= \iint_{I_{j}\times\mathbb{R}^{2}} P_{\tau}u(t)(e^{i(t-t_{j})\Delta}f_{r_{1}})(\overline{e^{i(t-t_{j})\Delta}f_{r_{2}}})(\overline{e^{i(t-t_{j})\Delta}f_{r_{3}}})dxdt  \\
&\leq  \big( \iint_{I_{j}\times\mathbb{R}^{2}}|P_{\tau}u|^{2}|e^{i(t-t_{j})\Delta}f_{r_{1}}|^{2}dxdt\big)^{1/2}\|e^{i(t-t_{j})\Delta}f_{r_{2}}\|_{L^{4}_{I_{j}\times\mathbb{R}^{2}}}\|e^{i(t-t_{j})\Delta}f_{r_{3}}\|_{L^{4}_{I_{j}\times\mathbb{R}^{2}}}
\ \ \text{by H\"{o}lder's inequality, }
\\
&\lesssim  \big( \iint_{I_{j}\times\mathbb{R}^{2}}|P_{\tau}u|^{2}|e^{i(t-t_{j})\Delta}f_{r_{1}}|^{2}dxdt\big)^{1/2}\|f_{r_{2}}\|_{2}\|f_{r_{3}}\|_{2}
\\
&\lesssim  \big( \iint_{I_{j}\times\mathbb{R}^{2}}|P_{\tau}u|^{2}|e^{i(t-t_{j})\Delta}f_{r_{1}}|^{2}dxdt\big)^{1/2}.
\end{align*}
That is, we have
\begin{align}
c\eta^{2}< \iint_{I_{j}\times\mathbb{R}^{2}}|P_{\tau}u|^{2}|e^{i(t-t_{j})\Delta}f_{r_{1}}|^{2}dxdt.
\end{align}
Applying Lemma 5 to $g=e^{-it_{j}\Delta}f_{r_{1}}$ with $\epsilon=\eta^{10}$, there are tubes $(Q_{s})_{1\leq s < S(\eta^{10})}$ such that
\begin{align}
\iint_{[\mathbb{R}^{3}\setminus\cup_{s} Q_{s}]\cap[I_{j}\times\mathbb{R}^{2}]}|P_{\tau}u|^{2}|e^{i(t-t_{j})\Delta}f_{r_{1}}|^{2}dxdt \leq \|P_{\tau}u\|^{2}_{L^{4}_{I_{j}\times \mathbb{R}^{2}}}\eta^{20} \ll c\eta^{2}.
\end{align}
Combining (4.15) and (4.16), there is a choice of $Q=\{(t,x):x+2t\xi_{0}\in K, t\in J\cap I_{j}\}\in(Q_{s})_{1\leq s < S(\eta^{10})}$, where $K$ is a square of sidelength $\frac{1}{\ell}$, and $J$ is an interval of length $\frac{1}{\ell^{2}}$, such that
\begin{align}
&\iint_{Q\cap [I_{j}\times \mathbb{R}^{2}]}|P_{\tau}u|^{2}|e^{i(t-t_{j})\Delta}f_{r_{1}}|^{2}dxdt  > \frac{\eta^{2}}{S(\eta^{10})}=: \eta_{1}>0.  \notag \\
\Rightarrow &\iint_{\{(t,x):x+2t\xi_{0}\in K,t\in J\cap I_{j}\}}|P_{\tau}u|^{4}dxdt > c\eta_{1}^{2},
\end{align}
by H\"{o}lder's inequality, a Strichartz estimate, and $f_{r_{1}}\in L^{2}(\mathbb{R}^{2})$.
Now observe that
\begin{align}
\|P_{\tau}u(t)\|_{L^{\infty}_{x}} \leq \|\widehat{P_{\tau}u}(t)\|_{L^{1}_{\xi}} \leq |\tau|^{1/2}\|u_{0}\|_{2} \leq C\ell.
\end{align}
For a small constant $\delta>0$, we apply (4.17), split up the integral, and peel out two factors of $\|P_{\tau}u(t)\|_{L^{\infty}_{x}}$ to find
\begin{align*}
c\eta_{1}^{2} &< \int_{J\cap[t_{j},t_{j+1}-\frac{\delta\eta_{1}^{2}}{\ell^{2}}]}\int_{x\in K-2t\xi_{0}}|P_{\tau}u(t)|^{4}dx   + \int_{J\cap[t_{j+1}-\frac{\delta\eta_{1}^{2}}{\ell^{2}},t_{j+1}]}\int_{x\in K-2t\xi_{0}}|P_{\tau}u(t)|^{4}dx    \\
&< \frac{1}{\ell^{2}}\sup_{t\in J\cap[t_{j},t_{j+1}-\frac{\delta\eta_{1}^{2}}{\ell^{2}}]}\int_{x\in K-2t\xi_{0}}|P_{\tau}u(t)|^{4}dx  +  \int_{t_{j+1}-\frac{\delta\eta_{1}^{2}}{\ell^{2}}}^{t_{j+1}}\int_{x\in K-2t\xi_{0}}|P_{\tau}u(t)|^{4}dx   \\
&\leq  C\sup_{t\in J\cap[t_{j},t_{j+1}-\frac{\delta\eta_{1}^{2}}{\ell^{2}}]}\int_{x\in K-2t\xi_{0}}|P_{\tau}u(t)|^{2}dx  + C\ell^{2}\frac{\delta\eta_{1}^{2}}{\ell^{2}}\|u_{0}\|_{2}  \\
&=  C\sup_{t\in J\cap[t_{j},t_{j+1}-\frac{\delta\eta_{1}^{2}}{\ell^{2}}]}\int_{x\in K-2t\xi_{0}}|P_{\tau}u(t)|^{2}dx  + \delta C\eta_{1}^{2}.
\end{align*}
Choosing $\delta$ sufficiently small, there is some $t\in [t_{j},t_{j+1}-\frac{\delta\eta_{1}}{\ell^{2}}]$ and a square $E_{1}=K-2t\xi_{0}$ of sidelength $\frac{1}{\ell}$ such that
\begin{align}
\int_{E_{1}}|P_{\tau}u(t)|^{2}dx > c\eta_{1}^{2}.
\end{align}
From $t\in[t_{j},t_{j+1}-\frac{\delta\eta_{1}}{\ell^{2}}]$, we have $t<t_{j+1}-\frac{\delta\eta_{1}^{2}}{\ell^{2}}<T^{*}-\frac{\delta\eta_{1}^{2}}{\ell^{2}}$, and therefore
\begin{align*}
side(E_{1})=\frac{1}{\ell}<C(T^{*}-t)^{1/2}.
\end{align*}
Furthermore
\begin{align*}
|P_{\tau}u(t)| \lesssim |u(t)|\ast \phi_{\ell},
\end{align*}
where $\phi_{\ell}(x)={\ell}^{2}\phi(\ell x)$ and $\phi$ is a smooth bump function supported on $[-1,1]^{2}$.  This gives
\begin{align}
|P_{\tau}u(t)|^{2} \lesssim |u(t)|^{2}\ast \phi_{\ell}.
\end{align}
Combining (4.19) and (4.20)
\begin{align*}
c\eta_{1}^{2} \lesssim \int_{E_{1}}|u(t)|^{2}\ast \phi_{\ell}dx = \langle|u(t)|^{2},\chi_{E_{1}}\ast \phi_{\ell}\rangle \lesssim \int_{E_{2}}|u(t)|^{2}dx ,
\end{align*}
for some square $E_{2}$ of sidelength $side(E_{2})=C(side(E_{1})) < C'(T^{*}-t)^{1/2}$.  Thus
\begin{align*}
\int_{E_{2}}|u(t)|^{2}dx > c\eta_{1}^{2}=: c' ,
\end{align*}
where $c'>0$ is a constant independent of $j$.  As this argument applies for each $j$, the expression $(1.11)$ follows.  This concludes the proof of Theorem 7.
\qed

\end{document}